\documentclass[12pt]{amsart}
\usepackage[cp850]{inputenc}
\usepackage[english]{babel}
\usepackage[curve]{xypic}%
\usepackage{graphics}    %
\RequirePackage{amssymb} 
\RequirePackage{amsopn}  
\RequirePackage{amsbsy}  
\RequirePackage{amsmath} 

\usepackage{enumerate}

\newtheorem{theorem}{Theorem}[section]
\newtheorem{proposition}[theorem]{Proposition}
\newtheorem{lemma}[theorem]{Lemma}
\newtheorem{corollary}[theorem]{Corollary}
\theoremstyle{definition}
\newtheorem{definition}[theorem]{Definition}
\newtheorem{example}[theorem]{Example}

\newtheorem{remark}[theorem]{Remark}

\numberwithin{equation}{section}

\newcommand{\balf}
 {\renewcommand{\theenumi}{(\alph{enumi})}\renewcommand{\labelenumi}{\theenumi}\begin{enumerate}
}
\newcommand{\ealf}
 {\end{enumerate}\renewcommand{\theenumi}{\arabic{enumi}}\renewcommand{\labelenumi}{\theenumi.}}
\newcommand{\bara}
 {\renewcommand{\theenumi}{(\arabic{enumi})}\renewcommand{\labelenumi}{\theenumi}\begin{enumerate}
}
\newcommand{\eara}
 {\end{enumerate}\renewcommand{\theenumi}{\arabic{enumi}}\renewcommand{\labelenumi}{\theenumi.}}
\newcommand{\brom}
 {\renewcommand{\theenumi}{(\roman{enumi})}\renewcommand{\labelenumi}{\theenumi}\begin{enumerate}
}
\newcommand{\erom}
 {\end{enumerate}\renewcommand{\theenumi}{\arabic{enumi}}\renewcommand{\labelenumi}{\theenumi.}}
\def\cuadro #1  {\framebox{\begin{minipage}[t]{123mm} #1\end{minipage}}}

\DeclareMathOperator{\Ker}     {Ker}%
\DeclareMathOperator{\length}  {length}%
\DeclareMathOperator{\botast}  {\bot{C^\ast}}%
\DeclareMathOperator{\botc}    {\bot{C}}%

\newcommand{\C}[2]{{#1}\cdot{#2}\cdot{#1}}
\newcommand{\eC}[1]{{e}\cdot{#1}\cdot{e}}

\usepackage{color}

\begin{document}
\title{Prime path coalgebras}
\author{P. Jara}%
 \address{Department of Algebra. University of Granada. 18071--Granada. SPAIN}%
 \email{pjara@ugr.es}%
 \urladdr{http://www.ugr.es/local/pjara}%
 \thanks{Supported by DGES MTM2004-08125, MTM2007-66666, FQM-266.}%
\author{L. Merino}
 \address{Department of Algebra. University of Granada. 18071--Granada. SPAIN}%
 \email{lmerino@ugr.es}%
\author{G. Navarro}
 \address{Department of Computer Sciences and AI. University of Granada.
18071--Granada. SPAIN}%
 \email{gnavarro@ugr.es}%
\author{J. F. Ruiz}
 \address{Department of Mathematics. University of Jaén. 23071--Jaén. SPAIN}%
 \email{jfruiz@ujaen.es}%

\date{\today}%

\maketitle

\begin{abstract}
We use prime coalgebras as a generalization of simple coalgebras, and observe that
prime subcoalgebras represent the structure of the coalgebra in a more efficient way
than simple coalgebras. In particular, in this work we focus our attention on the
study and characterization of prime subcoalgebras of path coalgebras of quivers and,
by extension, of prime pointed coalgebras.
\end{abstract}

\section{Preliminaries}

It is well known that every coalgebra, with separable coradical, is Morita--Takeuchi
equivalent to a subcoalgebra of a path coalgebra,
see\cite{Chin/Montgomery:1997,Nichols:1978}. From this result path coalgebras of
oriented graphs became important objects of study in the new developments in
Coalgebra Theory. Let us recall briefly some definition and facts involving them.

Following \cite{gabriel}, a \emph{quiver} (\emph{or oriented graph}), $G=(V,E,s,t)$,
is given by
 two sets
 $V$, the set of vertices, and
 $E$, the set of arrows, and
 two maps
 $s$, $t:E \rightarrow V$ providing each arrow $x$ with its source $s(x)$ and its
tail $t(x)$.
Sometimes we represent the arrow $x$ as $x:s(x)\to{t(x)}$.

A \emph{subquiver} of a quiver $G$ is a quiver $G'=(V',E',s',t')$ such that
$V'\subseteq{V}$, $E'\subseteq{E}$ and $s'=s\mid_{E'}$, $t'=t\mid_{E'}$.

A \emph{path} $p$ in $G$ is a finite sequence of arrows $p=x_1\cdots{x_n}$ in such a
way that
$t(x_{i})=s(x_{i+1})$ for every $i=1, \dots ,n-1$. In this case we set $s(p)=s(x_1)$
and
$t(p)=t(x_n)$. The \emph{length} of the path $p$ is the number of arrows which compose
it. For completeness, we consider vertices as trivial paths or paths of length
zero. For any trivial path $a$, we put $s(a)=a=t(a)$ and, for any path $p$ such that
$s(p)=a$ (resp. $t(p)=a$) we identify the concatenation $ap$ and $p$ (resp. $pa$ and
$p$).

A path of length $l\geq 1$ is called a \emph{cycle} whenever its source and its tail
coincide.

We also need the  notion of "\emph{unoriented path}" or walk. To each arrow
$x:a\rightarrow b$ in $G$, we associate a formal reverse $x^{-1}:b\rightarrow{a}$.
A \emph{walk} from a vertex $a$ to a vertex $b$ is a nonempty sequence of
arrows $x_1$, \ldots, $x_r$ such that, for every index $i$, there exists
$\varepsilon_i\in\{-1,1\}$ in such a way that
$x_1^{\varepsilon_1}\cdots{x_r^{\varepsilon_r}}$ is a path from $a$ to $b$.
Two vertices $a$ and $b$, of the quiver $G$, are said to be \emph{connected} if there
exists a walk from $a$ to $b$.
The quiver $G$ is called \emph{connected} if every two vertices of $G$ are connected.
The \emph{connected component} of a vertex $a\in{V}$ is the biggest connected subquiver
of $G$ containing $a$.
A quiver $G$ is said to be \emph{strongly connected} if, for each two vertices $a$
and $b$,
there exists a path in $G$ with source $a$ and tail $b$.

Let $p$ be a path in $G$. A path $q$ is a \emph{subpath} of $p$ if there exist paths
$p_1$ and
$p_2$ such that $p$ is the concatenation $p=p_1qp_2$.
If $q$ is a subpath of $p$, we write $q\preceq{p}$, and if, in addition, $q\neq{p}$,
we write $q\prec{p}$.
Alternatively, when $q$ is a subpath of $p$, we  say that $p$  \emph{passes
through} $q$. Note that it can happen that $p$ contains $q$, as subpath, more than
once; we denote $q\prec^n{p}$ if there exist paths $p_1$, \ldots , $p_{n+1}$ in such
a way that $p$ is equal to the concatenation $p_1 q p_2 q \cdots p_n q p_{n+1}$.

\vskip.5cm

In the following we assume that the reader is familiar with Coalgebra Theory.
Anyway we take \cite{ABE:1977,MONTGOMERY:1993,SWEEDLER:1969} as basic
references for coalgebras and comodules, and we refer the reader to them for undefined
terms.

\vskip.5cm

Let $G$ be a quiver and let $k$ be a field. The \emph{path coalgebra} of $G$ is the
$k$--vector space $PC(G)$, with basis the set of all paths, equipped with the following
comultiplication and counit:

For any vertex $a$:
$$
 \Delta(a)=a\otimes{a}\mbox{ and }\varepsilon(a)=1.
$$

For any non zero length path $p=x_1\cdots{x_n}$:
$$
\begin{array}{ll}
 \Delta(p)
 &=s(p)\otimes{p}+\sum_{i=1}^{n-1}{x_1}\cdots{x_{i}\otimes{x_{i+1}}\cdots{x_n}}+p\otimes{t(p)}\\
 &=\sum_{p_1p_2=p}p_1\otimes{p_2},\mbox{ and }\\
 \varepsilon(p)
 &=0.
\end{array}
$$

As a consequence of this definition we have the following facts:
\begin{enumerate}[(1)]
\item
$(PC(G),\Delta,\varepsilon)$ is a pointed coalgebra, being the simple subcoalgebras
generated by the vertices.
\item
Any pointed coalgebra $C$ is isomorphic to a subcoalgebra of a certain path coalgebra,
see \cite{Chin/Montgomery:1997,Woodcock:1997}.
\item
The path coalgebra $PC(G)$ can be also constructed as $T_{k V}(kE)$, the cotensor
coalgebra over $kV$ defined by $kE$, see
\cite{Nichols:1978,Jara/Llena/Merino/Stefan:unp}.
\end{enumerate}

As it is showed in the literature, simple subcoalgebras only control the vertices of
the quiver,
however they do not control the arrows. See Example \ref{ex:15} below. For this
reason we are interested in a generalization of simple subcoalgebras: the prime
subcoalgebras. To introduce them, let us first recall the concept of wedge product.

Let $A$ and $B$ be two subcoalgebras of a coalgebra $C$. The \emph{wedge product}
\cite{SWEEDLER:1969},
$A\wedge^CB$, of $A$ and $B$ in $C$ is defined as:
\[
\begin{array}{lll}
 A\wedge^CB
 &=\Ker(C\stackrel{\Delta}{\longrightarrow}C\otimes{C}
   \stackrel{pr\otimes{pr}}{\longrightarrow}\frac{C}{A}\otimes\frac{C}{B})\\
 &=\Delta^{-1}(C\otimes{B}+A\otimes{C})\\
 &=(A^{{\botast}}B^{{\botast}})^{{\botc}}.
\end{array}
\]
It is known that $A\wedge^CB$ is a subcoalgebra of $C$ containing $A+B$ and, in
general, it happens that $A\wedge^CB\neq{B\wedge^CA}$.

\begin{definition}\label{de:030527}
A coalgebra $C$ is said to be \emph{prime} if, for any subcoalgebras $A$ and $B$ of $C$
such that $C=A\wedge^CB$, we have either $C=A$ or $C=B$.
\end{definition}

\begin{lemma}\label{le:030527a}
Let $D$ be a subcoalgebra of a coalgebra $C$, the following statements are
equivalent:
\begin{enumerate}[(a)]
\item
$D$ is a prime coalgebra.
\item
For any subcoalgebras $A$ and $B$ of $C$ such that $D\subseteq{A\wedge^C B}$, we have
either $D\subseteq{A}$ or $D\subseteq{B}$.
\end{enumerate}
\end{lemma}
\begin{proof}
(a) $\Rightarrow$ (b). Let $A$ and $B$ be subcoalgebras of $C$ such that
$D\subseteq{A\wedge^C B}$, then
$D=(A\wedge^CB)\cap{D}\subseteq(A\cap{D})\wedge^D(B\cap{D})$. Hence either $D=A\cap{D}$
or $D=B\cap{D}$.
\newline
(b) $\Rightarrow$ (a). Let $X$ and $Y$ be subcoalgebras of $D$ such that
$D=X\wedge^D Y=
(X\wedge^C Y) \cap{D}$, then $D\subseteq{X\wedge^C Y}$, and  we have either $D=X$ or
$D=Y$.
\end{proof}

\begin{remark} \label{tesistak}
In  Takeuchi's thesis \cite[1.4.2]{Takeuchi:1974} appears the concept of
\emph{coprime subcoalgebra} of a cocommutative coalgebra: a subcoalgebra $D$ of $C$ is
said coprime if it satisfies the condition (b) of the previous Lemma. He proved that a
subcoalgebra $D$, of a cocommutative coalgebra $C$, is a coprime subcoalgebra of $C$ if
and only if $D^{\botast}$ is a prime ideal of the commutative algebra $C^*$.
Actually, his proof is also valid for non necessarily cocommutative coalgebras.
\end{remark}

Let us recall that a coalgebra $C$ is called \emph{indecomposable} if there are no two
non trivial proper subcoalgebras $D_1$ and $D_2$ such that $C=D_1\oplus{D_2}$. It is
well
known, see \cite{Montgomery:1995}, that the path coalgebra $PC(G)$ is indecomposable if
and only if the quiver $G$ is connected.

\begin{lemma}
The following statements hold.
\begin{enumerate}[(1)]
 \item Every simple coalgebra is prime.
 \item Any prime coalgebra is indecomposable.
 \item Every finite-dimensional prime coalgebra is simple
\end{enumerate}
\end{lemma}
\begin{proof}
The first assertion is trivial. For the second one, if $C$ is a prime
coalgebra and $C=A\oplus{B}$ is a direct sum of two subcoalgebras $A$, $B\subseteq{C}$,
then $C=A+B\subseteq{A\wedge{B}}$. Hence either $C=A$ or $C=B$. Finally, suppose
that $D$
is prime and finite-dimensional and denote by $R$ the coradical of $D$. Then $D=\wedge
^{\infty}R$, see \cite{SWEEDLER:1969}, and being $D$ finite-dimensional, $D=\wedge
^{n}R$
for some $n$. Thus $D=R$, and it is cosemisimple. By (2), since $D$ is indecomposable,
then $D$ is simple.
\end{proof}

Let us consider the following example in which we obtain that prime coalgebras give us
more information than simple coalgebras in order to describe a coalgebra.

\begin{example}
Let $L$ be a finite-dimensional Lie algebra generated by $x_1$, \ldots, $x_n$, and
consider the universal enveloping algebra, say $C=U(L)$. It is well known that $C$
has a
coalgebra structure in which $1$ is the unique group like element and every element
$x_i$ is primitive, i.e. $\Delta(x_i)=x_i \otimes 1 + 1 \otimes x_i$, for any index
$i$.
Thus $k1$ is the only simple subcoalgebra of $C$. Nevertheless it is no difficult to
prove that for every index $i=1,\ldots,n$ the vector space $D_i$, generated by all the
powers of $x_i$, is a prime subcoalgebra of $C$ and $C=D_1+\cdots+{D_n}$.
\end{example}

\section{Subcoalgebras of a Path coalgebra}\label{se:2}

Throughout this section we consider a quiver $G$. We study properties relative
to elements of subcoalgebras of $PC(G)$ and the relationship with paths and vertices.

\begin{proposition}\label{pr:030411}
Let $D$ be a subcoalgebra of $PC(G)$, and let $p$ be a path in $D$. If $q\preceq{p}$,
then $q$ belongs to $D$.
\end{proposition}
\begin{proof}
Indeed, let us denote $p=x_1\cdots{x_r}$ and $q=x_i\cdots{x_{i+s}}$. Then in
$\Delta(p)$
the summand
 $x_1\cdots{x_{i-1}}\otimes{x_i}\cdots{x_r}$
appears. Since $D$ is a subcoalgebra of $PC(G)$, then we have
 $x_i\cdots{x_r}$ is in $D$.
In an analogous way, from $x_i\cdots{x_r}\in{D}$, we may deduce that
$x_i\cdots{x_{i+s}}$
belongs to $D$.
\end{proof}

The following is another closure property for elements in a subcoalgebra.

\begin{lemma}\label{le:22}
Let $D$ be  a subcoalgebra of $PC(G)$, and let $d=\sum_{i=1}^s\lambda_ip_i$ be a non
zero
element in $D$. Given a path $p$, let us assume that $\{1, \dots , r\}$,  $r \leq s$ is
\emph{the set of all} indices $i$ such that there exists a subpath $p'_{i}$ of $p_i$
verifying $p_{i}=pp'_{i}$. Then  $\sum_{j=1}^r\lambda_i p'_{i} \in{D}$.
\end{lemma}
\begin{proof}
From the hypothesis we may see that in $\Delta(d)$ appears a summand
$p\otimes(\sum_{i=1}^s\lambda_i p'_{i})$ and that $p$ does not appear in the first
component of the remaining summands. Let us consider a linear map $f\in (PC(G))^*$
defined by $f(p)=1$ and $f(q)=0$ for any path $q\neq{p}$. Then
$d\cdot{f}=f(p)\sum_{i=1}^s\lambda_ip'_{i}=\sum_{i=1}^s\lambda_ip'_{i}$ is an
element of
$D$.
\end{proof}

Any element $d\in{PC(G)}$ can be written uniquely as a $k$--linear combination of
paths,
say $d=\sum_{i=1}^s\lambda_ip_i$.
If $p=x_1\cdots{x_r}$ is a path, let us define
$V(p)=\{s(x_1),\ldots,s(x_{r}),t(x_r)\}$,
and extend this definition to elements of $PC(G)$: for any $d\in{PC(G)}$ such that
$d=\sum_{i=1}^s\lambda_ip_i$, with $\lambda_i\neq0$, we define
$V(d)=\cup\{V(p_i)\mid\;i=1,\ldots,s\}$.
Again we may extend this definition to any non empty subset $X\subseteq{PC(G)}$ by
setting $V(X)=\cup\{V(d)\mid\;d\in{X}\}$.

With this notation we may state and prove the following result.

\begin{proposition}\label{pr:22}
If $D$ is a subcoalgebra of $PC(G)$, then $V(D)\subseteq{D}$.
\end{proposition}
\begin{proof}
Let us consider a vertex $a\in{V(D)}$, then there is some element
$d=\sum_{i=1}^s\lambda_ip_i\in{D}$ such that $p_1=x_1\cdots{x_r}$ and $a=s(x_j)$, for
some $j=1,\ldots,r-1$. We treat the following two cases:
\newline\hspace*{.5cm}
(1) $a=s(x_1)$. Then we have the decomposition
$$
 \Delta(d)=a\otimes{d}+\mbox{\textit{other terms}}.
$$
If we define a linear map $f:PC(G)\longrightarrow{k}$ as $f(d)=1$ and $f(p)=0$
for any path $p$ such that
$$
 \length(p)<\max\{\length(p_i)\mid\;i=1,\ldots,t\}
$$
then $f\cdot{d}=a\in{D}$.
\newline\hspace*{.5cm}
(2) $a=s(x_j)$ for some $j=2,\ldots,r$. Then we consider the path $x_j\cdots{x_r}$. By
Lemma \ref{le:22}, we may assume that $a$ is in case (1), for a new non zero element in
$D$. Therefore $a\in{D}$.
\end{proof}

Next we obtain a key tool in this paper.

\begin{theorem}\label{th:22}
Let $D$ be  a subcoalgebra of $PC(G)$, then there exists a basis $B$ of $D$ such that
every basic element in $B$ is a linear combination of paths with common source and
common tail.
\end{theorem}
\begin{proof}
Let $d\in{D}$. Consider the decomposition $d=d_1+\cdots+d_t$, where each $d_i$ is a
linear combination of paths with common source and common tail. Let us prove that
$d_i\in{D}$ for any $i=1,\ldots,t$. Fix an index $i$ and assume that the paths in $d_i$
start at $a$ and end at $b$. For any vertex $v\in{V}$ we may define two sets of indices
as follows:
$$
 H^0_v=\{h\mid\;d_h\mbox{ is a linear combination of paths $p$}\mbox{ starting at }v\},
$$
$$
 H^1_v=\{h\mid\;d_h\mbox{ is a linear combination of paths $p$}\mbox{ starting at }v\}.
$$
Then there exists a decomposition $d=\sum_{v\in V}\sum_{h\in{H^0_v}}d_h$. Hence
$$
 \Delta(d)=\sum_{v\in{V}}\sum_{h\in{H^0_v}}v\otimes{d_h}+\mbox{\textit{other terms}}.
$$
We consider the linear map $f_v:PC(G)\longrightarrow{k}$ defined by $f_v(v)=1$ and
$f_v(p)=0$, for any path $p\neq{v}$. Then $\sum_{h\in{H^0_v}}d_h=d\cdot{f_v}\in{D}$. In
the same way, for any $w\in{V}$, we have $\sum_{h\in{H^1_w}}d_h=f_w\cdot{d}\in{D}$.
That
is, $\sum_{h\in{H^0_v\cap{H^1_w}}}d_h=f_w\cdot{d}\cdot{f_v}\in{D}$. Take $w=b$, $v=a$
and then $d_i=\sum_{h\in{H^0_a\cap{H^1_b}}}d_h=f_b\cdot{d}\cdot{f_v}\in{D}$.
\end{proof}

Let $D$ be a subcoalgebra of a path coalgebra $PC(G)$. If $D$ is the path coalgebra
$PC(G')$ associated to a subquiver $G'$ of $G$, then we  say that $D$ is a
 \emph{path subcoalgebra}
of $PC(G)$. It is interesting to relate subcoalgebras and path subcoalgebras of
$PC(G)$,
since, even when the path subcoalgebra is far away from the given subcoalgebra, we
may see
that it contains some relevant information.

The simplest method to define such subcoalgebra is the following:
\begin{enumerate}[(1)]
\item
We consider $P(D)$, the set of paths defined by
\[
\begin{array}{ll}
 P(D):=\{p\in{PC(G)}\mid\;
 &p\mbox{ is a  path which appears non }\\
 &\mbox{trivially in some element of }D\},
\end{array}
\]
\item
We denote $E(D):=E\cap{P(D)}$, the set of arrows in $G$ that belong to $P(D)$, and we
have $V(D)=V\cap{P(D)}$.
\item
We denote by $G(D)$ the quiver $(V(D),E(D))$, and call it the \emph{quiver associated}
to the subcoalgebra $D$.
\end{enumerate}

In the following Example we show how different could be the coalgebras $D$ and
$PC(G(D))$.

\begin{example}\label{ex:15}
We consider the quiver $G$ given by the picture:
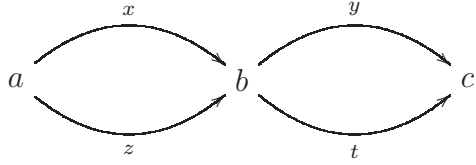
\[
\xy
 (0,0)*+{a}="a",
 (30,0)*+{b}="b",
 (60,0)*+{c}="c"
 \ar@/^4ex/ "a";"b" ^x
 \ar@/^4ex/ "b";"c" ^y
 \ar@/^-4ex/ "a";"b" _z
 \ar@/^-4ex/ "b";"c" _t
\endxy
\]
And let $D$ be the subcoalgebra generated by $\alpha=xy+xt+zy+zt$. A $k$-basis of
$D$ is
\[
\{a,b,c,x+z,y+t,\alpha\},
\]
as $\Delta(\alpha)=a\otimes{\alpha}+(x+z)\otimes(y+t)+\alpha\otimes{c}$. The coalgebra
$D$ satisfies:
\begin{itemize}
\item
$V(D)=\{a,b,c\}$,
\item
$E(D)=\{x,y,z,t\}$, and
\item
$P(D)=\{a,b,c,x,y,z,t,xy,xt,zy,zt\}$.
\end{itemize}
Therefore
 $E(D)\nsubseteq{D}$,
 $D\subsetneqq{k\cdot{P(D)}}={PC(G(D))}$ and
 $\dim(D)=6\neq11=\dim(PC(G(D)))$.
\end{example}

\begin{remark}
Let $D$ be a subcoalgebra of $PC(G)$. By Lemma~\ref{le:22} and Theorem~\ref{th:22}, for
every path $p\in{P(D)}$, there exists a linear combination of paths with common source
and common tail, $\sum_{i=1}^t\lambda_ip_i\in{D}$, such that $p=p_i$ for some index
$i$.
\end{remark}

The following Proposition shows that, for any subcoalgebra $D\subseteq{PC(G)}$, the
subcoalgebra $PC(G(D))$ is the smallest path subcoalgebra of $PC(G)$ containing $D$.

\begin{proposition}
Let $A$ and $B$ be subcoalgebras  of $PC(G)$, then the following
statements hold.
\begin{enumerate}[(1)]
\item $P(A)$ is closed under subpaths.
\item $A\subseteq{k\cdot{P(A)}}\subseteq{PC(G(A))}$ is a tower of subcoalgebras of $C$.
\item $A$ has a basis constituted by paths if and only if $A=k\cdot{P(A)}$.
\item If $A\subseteq{B}$, then $P(A)\subseteq{P(B)}$.
\item $P(A\wedge{B})\subseteq{P(A)}\wedge{P(B)}$ and $k\cdot{P(A\wedge{B})}$ is a
subcoalgebra of $k\cdot{P(A)}\wedge{k\cdot{P(B)}}$.
\item $k\cdot{P(A+B)}=k\cdot{P(A)}+k\cdot{P(B)}$.
\item $P((G(D))$ is the smallest path subcoalgebra of $PC(G)$ containing $D$.
\end{enumerate}
\end{proposition}

The proof is straightforward.

We also need to study the behavior of the paths in $PC(G)$ with respect to
the wedge product of subcoalgebras. We state the following, possibly known, results.

\begin{proposition}
Let $A$ and $B$ be subcoalgebras  of $PC(G)$,  then the following statements hold:
\begin{enumerate}[(1)]
\item
$V(A\wedge{B})=V(A)\cup{V(B)}$.
\item
For any edge $x$ such that $s(x)\in{A}$ and $t(x)\in{B}$, we have $x\in{A\wedge{B}}$.
\item
For any path $p_1\in{A}$ and any path $p_2\in{B}$ such that $t(p_1)=s(p_2)$, we have
that $p_1p_2\in{A\wedge{B}}$.
\item
For any path $p_1\in{A}$, $p_2\in{B}$ and any edge $x\in{E}$ such that $t(p_1)=s(x)$
and
$t(x)=s(p_2)$ we have $p_1xp_2\in{A\wedge{B}}$.
\item
$E(A\wedge{B})=E(A)\cup{E(B)}\cup\{x\in{E}\mid\;s(x)\in{V(A)}\mbox{ and
}t(x)\in{V(B)}\}$.
\end{enumerate}
\end{proposition}

The proof is straightforward.

As a consequence of these two Propositions we obtain the following characterization of
coidempotent subcoalgebras of path coalgebras. We recall that a subcoalgebra $D$ of a
coalgebra $C$ is said to be a \emph{coidempotent subcoalgebra} if $D\wedge^C{D}=D$.

\begin{theorem}\label{th:030331}
Let $D$ be a subcoalgebra of $PC(G)$, then the following statements are equivalent:
\begin{enumerate}[(a)]
\item
$D$ is coidempotent.
\item
$D$ is the path coalgebra of the subquiver $G(D)$ and
$E(D)=\{x\in{E}\mid\;s(x),t(x)\in{V(D)}\}$.
\end{enumerate}
\end{theorem}
\begin{proof}
(a) $\Rightarrow$ (b). Let $D\subseteq{PC(G)}$ be a coidempotent subcoalgebra, then
$V(D)\subseteq{D}$ and, for any arrow $x$ such that $s(x)$, $t(x)\in{V(D)}$, we have
$\Delta(x)=s(x)\otimes{x}+x\otimes t(x)$. Hence $x\in{D\wedge{D}}=D$, i.e.,
$x\in{E(D)}$. Otherwise, if $p$ is a path of length $t$ in the quiver $(V(D),E(D))$,
then $p\in{D^{\wedge{t}}}=D$. Hence $PC(V(D),E(D))\subseteq{D}$ and they are equal.
\newline
(b) $\Rightarrow$ (a). We have $V(D)=V(D\wedge D)$, therefore given an arbitrary
element
$x=\sum_{i=1}^t\lambda_ip_i$ in $D\wedge D$ we have, for every $i=1, \dots ,t$, $V(p_i)
\subseteq D$ and then $p_i\in D$ for each $i$. Thus $x\in D$.
\end{proof}

\begin{remark}
Combining \cite[Proposition 3.8]{Woodcock:1997} and \cite[Theorem
4.5]{Nastasescu/Torrecillas:1996} we obtain a bijective correspondence between
coidempotent subcoalgebras of an arbitrary coalgebra $C$ and subsets of a fixed set of
representatives of simple right $C$--comodules. Thus Theorem~\ref{th:030331} shows
explicitly this correspondence in the case of path coalgebras.
\end{remark}

\section{Prime subcoalgebras of a path coalgebra}

In this section we apply the results of Section~\ref{se:2} in order to characterize
prime subcoalgebras of a path coalgebra. We start with a Theorem which
provides information about prime subcoalgebras of $PC(G)$ and their
elements.

\begin{theorem}\label{th:31}
Let $D$ be a prime subcoalgebra of $PC(G)$. For any path $q\in{P(D)}$ and any positive
integer $n$, there exists a cycle $c\in{P(D)}$ such that $q\prec^nc$.
\end{theorem}
\begin{proof}
First we prove that there exists a cycle passing through $q$. We consider the vector
space $A$ with basis
\[
 \{p\mid\;p\mbox{ is a path in }G\mbox{ such that }q\npreceq{p}\}.
\]
It is clear that $A$ is a subcoalgebra of $PC(G)$, and  $D\not\subseteq{A}$.
\newline
We claim that, \emph{if a path $p$ satisfies $p \notin \wedge ^n A$, then $q\prec^n
p$.}
Indeed, let us consider $n=2$; if $p \notin{A}\wedge{A}$, then $p\notin{A}$ so
$q\prec{p}$. If $q\not\prec^2 q$, then $p=r_1 q r_2$ for some paths $r_1$, $r_2\in{A}$.
Hence $\Delta (p) \in A \otimes C + C \otimes A$.
Inductively, if $p \notin \wedge^{n+1} A$, then $p \notin \wedge^{n}A$ so, by the
induction hypothesis, $q\prec^n p$. If $q\not\prec^{n+1} q$, then $p$ can be written as
$p=r_1qr_2\cdots{r_nqr_{n+1}}$ for some paths  $r_1$, \ldots, $r_{n+2}\in{A}$. Hence
$\Delta (p) \in A \otimes C + C \otimes \wedge^{n+1}A$.
\newline
Since $D$ is prime such that $D \not\subseteq A$, we obtain that
$D\nsubseteq\wedge^{n+1}A$. So, there exist some $\alpha\in{A}$ such that
$\alpha\notin\wedge^{n+1}A$. In particular, we obtain that there exists a path $p$ in
$P(D)$ such that $p\notin\wedge^{n+1}A$ (one of the paths appearing in the
expression of
$\alpha$), therefore $q\prec^{n+1} p$. If $p=r_1 q r_2 \dots r_n q r_{n+1} q r_{n+2}$
for some paths $r_1$, \ldots, $r_{n+2}\in{A}$, there exists a subpath $c$ of $p$, which
is a cycle, such that $c \prec^n p$, namely
 $c=qr_2\cdots{r_nqr_{n+1}}$.
\end{proof}

Let us prove the following consequence:

\begin{corollary}
Let $D$ be a prime subcoalgebra of $PC(G)$. Let $\sum_{i=1}^s\lambda_ip_i\in{D}$, where
$p_1$, \ldots, $p_s$ are pairwise different paths, then, for any positive integer
$n$, there exists a path $q$ in $P(D)$ such that $p_i \prec^n q$ for all
$i=1,\ldots,s$.
\end{corollary}
\begin{proof}
For simplicity we may assume $s=2$, being analogous the proof in the general case. We
consider the vector spaces $A$, with basis
$$
 \{q\mid\;q\mbox{ is a path in }G\mbox{ such that }p_1\npreceq{q}\},
$$
and $B$, with basis
$$
 \{q\mid\; q\mbox{ is a path in }G\mbox{ such that }p_2\npreceq{q} \}.
$$
Then $A$ and $B$ are subcoalgebras of $PC(G)$, $D\nsubseteq A$ and $D\nsubseteq{B}$.
Hence $D\nsubseteq{A\wedge{B}}$, in particular $D\nsubseteq A+B$. Let $\alpha\in{D}$
such that $\alpha\notin{A+B}$. Then there exists a path $p$ in $P(D)$, that appears in
$\alpha$, such that $p\notin{A\cup{B}}$. That is, $p_1\prec{p}$ and $p_2\prec{p}$.
By applying the previous Theorem to the path $p$, the result follows.
\end{proof}

In order to characterize prime subcoalgebras of a path coalgebra, the simplest case
appears when we consider a path subcoalgebra. Let us start the study of this case with
the following Example.

\begin{example}
Let $G$ be a quiver with three arrows $x_1$, $x_2$ and $x_3$ such that
 $t(x_1)=s(x_2)=:a_1$,
 $t(x_2)=s(x_3)=:a_2$ and
 $t(x_3)=s(x_1)=:a_3$,
i.e., they form a cycle of length three.
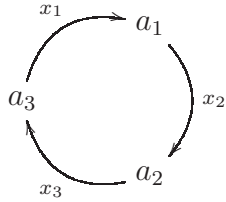
\[
\xy
 (0,10)*+{a_3}="a",
 (17,0)*+{a_2}="b",
 (17,20)*+{a_1}="c"
 \ar@/^3.1ex/ "a";"c" ^{x_{1}}
 \ar@/^3.1ex/ "c";"b" ^{x_{2}}
 \ar@/^3.1ex/ "b";"a" ^{x_{3}}
\endxy
\]
Let $D=PC(G)$. As vector space, $D$ is generated by the set
\[
 \{q\mid\;q\mbox{ is a subpath of }(x_1x_2x_3)^n\mbox{ for some }n\geq1\},
\]
where, for any $n\geq1$, $(x_1x_2x_3)^n$ is defined recursively in the usual way.

\vskip.25cm

We claim that \emph{$D$ is prime}. Let $A$ and $B\subseteq{D}$ be two subcoalgebras and
$D\subseteq{A\wedge{B}}$. We assume $(x_1x_2x_3)^n\notin{B}$ for some $n\geq1$. We
proceed as follows: since $(x_1x_2x_3)^{n+h}\in{D}$ and
\[
 \Delta((x_1x_2x_3)^{n+h})=(x_1x_2x_3)^h\otimes(x_1x_2x_3)^n+\mathit{other\; terms},
\]
$(x_1x_2x_3)^h\in{A}$ as $(x_1x_2x_3)^n\notin{B}$. This is true for every $h\geq1$,
hence $D\subseteq{A}$.
\end{example}

In the above example we obtain that $D$ is the path coalgebra of a cycle. Therefore one
could wonder if any prime path coalgebra must be the path coalgebra of a set of cycles.
In order to explore this question we need to deepen into the graph structure.

\begin{theorem}\label{th:33}
Let $G$ be a quiver, then the following statements are equivalent:
\begin{enumerate}[(a)]
\item
$PC(G)$ is prime.
\item
$G$ is strongly connected.
\end{enumerate}
\end{theorem}
\begin{proof}
(a) $\Rightarrow$ (b). %
Note that, if $PC(G)$ is prime, it is indecomposable. Hence $G$ is connected, now, for
any $a$, $b\in{V}$, there exists a walk
$x_1^{\varepsilon_1}\cdots{x_r^{\varepsilon_r}}$
from $a$ to $b$.
For every index $i$ such that $\varepsilon_i=-1$, by Theorem~\ref{th:31}, we can
consider a cycle $c_i=p_ix_iq_i$ passing through $x_i$. Then $q_ip_i$ is a path from
$t(x_i)$ to $s(x_i)$. So, we obtain a path from $a$ to $b$.
\newline
(b) $\Rightarrow$ (a). %
Let us assume $PC(G)$ is not prime, then there are two subcoalgebras $A$ and $B$ such
that $PC(G)=A\wedge{B}$, $A\subsetneqq{PC(G)}$ and $B\subsetneqq{PC(G)}$. Since
$A\subsetneqq{PC(G)}$, there exists $d\in{PC(G)}\setminus{A}$. If we assume
$d=\sum_{i=1}^t\lambda_ip_i$, for some paths $p_1$, \ldots, $p_t$, then there is some
$p_i\in{PC(G)}\setminus{A}$. In the same way we show that there is some
$p_j\in{PC(G)}\setminus{B}$. Since $G$ is strongly connected, there is a path $p$ such
that $s(p)=t(p_i)$ and $t(p)=s(p_j)$. Hence $p_ipp_j\in{PC(G)}=A\wedge{B}$. Now
\[
 \Delta(p_ipp_j)=p_ip\otimes{p_j}+p_i\otimes{pp_j}+\mathit{other\; terms}.
\]
Hence $p_ip\in{A}$ as $p_j\notin{B}$, therefore $p_i\in{A}$, which is a contradiction.
As a consequence, $PC(G)$ is prime.
\end{proof}

Now the problem is to determine all prime subcoalgebras of $PC(G)$. Let us start
with an
example.

\begin{example}
We consider the quiver $G$ given by
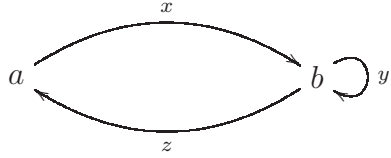
\[
\xy
 (0,0)*+{a}="a",
 (40,0)*+{b}="b"
 \ar@/^4ex/ "a";"b" ^x
 \ar@/_-4ex/ "b";"a" ^z
 \ar@(ur,dr) "b";"b" ^y
\endxy
\]
and $D$ the subcoalgebra of $PC(G)$ generated by $\{(xyz)^n\mid\;n\in\mathbb{N}\}$. It
is clear that $D$ is prime. The coalgebra $k\cdot{P(D)}$ is also prime and it is not a
path coalgebra, i.e., $k\cdot{P(D)}\neq{PC(G(D))}$.
\end{example}

We may give a characterization of certain prime subcoalgebras of a path coalgebra.

\begin{theorem}\label{le:14a}
Let $D$ be a subcoalgebra of $PC(G)$ such that either $D$ has a basis constituted by
paths, or $P(D)$ is closed under concatenation. Then the following statements are
equivalent:
\begin{enumerate}[(a)]
\item
$D$ is prime.
\item
If $p_1$, $p_2\in{P(D)}$, then there exists $q\in{P(D)}$ such that $p_1$,
$p_2\preceq{q}$.
\end{enumerate}
\end{theorem}
\begin{proof}
(a) $\Rightarrow$ (b). It is a consequence of Theorem~\ref{th:31} above.
\newline
(b) $\Rightarrow$ (a). %
Case (1). \textsl{$D$ has a  basis constituted by paths.}
\newline
Let $A$, $B\subseteq{PC(G)}$ be two subcoalgebras such that $D\nsubseteq A$ and
$D\nsubseteq B$. Under our hypothesis, we may obtain paths $p\in{D}\setminus{A}$ and
$h\in{D}\setminus{B}$. By Theorem~\ref{th:31}, there exists a path $q$ passing through
$p$ and $h$. Actually we may assume $q=pch$ for some path $c$. Then $pch\in{D}$ and
$pch\notin{A\wedge{B}}$, so $D\nsubseteq{A\wedge{B}}$.
\newline
Case (2). \textsl{$P(D)$ is closed under concatenation.}
\newline
Let $D\subseteq{A\wedge{B}}$ and suppose that $A\subsetneqq{D}$ and $B\subsetneqq{D}$.
There exists a path $p_0\in{P(D)}\setminus{P(A)}$ and a linear combination of paths
with
non zero coefficients $\sum_{i=0}^t\lambda_ip_i\in{D}\setminus{A}$ where the paths
$p_i$
have common source and common tail. Similarly, there exist $q_0$ and $\sum_j\mu_jq_j$
with the same properties with respect to $B$. Hence, by hypothesis, there exists
$h_0\in{P(D)}$ such that $\sum_k\eta_kh_k\in{D}$. Let us consider the element
$\alpha:=\sum_{ikj}\lambda_i\eta_k\mu_jp_ih_kq_j$, then
$\alpha\in{D}\subseteq{A\wedge{B}}$. By construction,
$\Delta(\sum_{ikj}\lambda_i\eta_k\mu_jp_ih_kq_j)\notin{A\otimes{PC(G)}+PC(G)\otimes{B}}$,
which is a contradiction.
\end{proof}

\begin{corollary}\label{co:14}
Let $D$ be a prime subcoalgebra of $PC(G)$, then $PC(G(D))$  and $k\cdot{P(D)}$ are
also
prime.
\end{corollary}
\begin{proof}
It is a direct consequence of Theorems~\ref{th:33} and~\ref{le:14a}.
\end{proof}

Unfortunately the converse is not true, as the next example shows.

\begin{example}
Let us consider the quiver $G$ given by
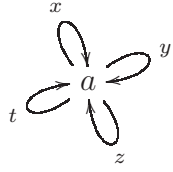
\[
\xy
 (0,0)*+{a}="a"
 \ar@(ul,u)^x "a";"a"
 \ar@(ur,r)^y "a";"a"
 \ar@(dr,d)^z "a";"a"
 \ar@(dl,l)^t "a";"a"
\endxy
\]
and $A$ and $B$ the coalgebras generated by $\{a+y,z+t\}$ and $\{y+z\}$,
respectively. If
we take $D=A\wedge{B}$ then $D$ is not prime, nevertheless
$k\cdot{P(D)}=PC(G(D))=PC(G)$
is a prime coalgebra.
\end{example}

Those results can be also applied to study prime subcoalgebras of an arbitrary pointed
coalgebra. Indeed, if $C$ is a pointed coalgebra, we may consider the quiver $G$ whose
vertices are the group like elements of $C$ and whose arrows are given by the skew
primitive elements. Then $C$ is embedded, as a subcoalgebra, into the path coalgebra
$PC(G)$. See \cite{Woodcock:1997}.

Nevertheless, the problem of characterizing prime subcoalgebras of path coalgebras is
still open. We study this problem from a different point of view in the next section.

\section{Prime coalgebras and idempotent elements.}

Throughout this section $C$ will be a coalgebra and $PC(G)$ a path coalgebra. Let
$e\in{C^\ast}$ be a non zero idempotent element, the vector space $e\cdot{C}\cdot{e}$
has a coalgebra structure (see \cite{Cuadra/Gomez:2002,Radford:1982}) given  by:
\[
 \Delta(\eC{x})=\sum(\eC{x_1})\otimes(\eC{x_2}),
 \quad
 \varepsilon(\eC{x})=e(\eC{x})\mbox{ for any }x\in{C}.
\]

There always exists a linear map $\phi:{C}\longrightarrow\eC{C}$ defined by
$\phi(x)=\eC{x}$ for any $x\in{C}$. This map $\phi$ satisfies
\[
 \Delta_{\eC{C}}\phi(c)=(\phi\otimes\phi)\Delta_C(c)\quad\mbox{for any }c\in{C}.
\]
As a consequence, we obtain the following result:

\begin{lemma}
\begin{enumerate}[(1)]
\item
If $A\subseteq{C}$ is a subcoalgebra, then $\phi(A)\subseteq{\eC{C}}$ is a
subcoalgebra.
\item
If $H\subseteq\eC{C}$ is a subcoalgebra, then $\phi^{-1}(H)\subseteq{C}$ is a
subcoalgebra.
\item
If $H$, $L\subseteq\eC{C}$ are subcoalgebras, then
$\phi^{-1}(H\wedge{L})\subseteq\phi^{-1}(H)\wedge\phi^{-1}(L)$.
\item
If $A$, $B\subseteq{C}$ are subcoalgebras, then
$\phi(A\wedge{B})\subseteq\phi(A)\wedge\phi(B)$.
\end{enumerate}
\end{lemma}
\begin{proof}
(1). %
For any $a\in{A}$, we have:
\[
 \Delta_{\eC{C}}(\phi(a))
 =(\phi\otimes\phi)\Delta_C(a)
 =\sum\phi(a_1)\otimes\phi(a_2)
 \in\phi(A)\otimes\phi(A).
\]
(2). %
For any $x\in\phi^{-1}(H)$, we have $\phi(x)\in{H}$, then
$(\phi\otimes\phi)\Delta_C(x)=\Delta_{\eC{C}}(\phi(x))\in{H\otimes{H}}$, hence
$\Delta_C(x)\in(\phi\otimes\phi)^{-1}(H\otimes{H})=\phi^{-1}(H)\otimes\phi^{-1}(H)$.
\newline (3). %
Let $x\in\phi^{-1}(H\wedge{L})$, then $\phi(x)\in{H\wedge{L}}$, or equivalently
$\Delta_{\eC{C}}(\phi(x))=H\otimes\eC{C}+\eC{C}\otimes{L}$, hence
$\Delta_C(x)\in\phi^{-1}(H)\otimes{C}+C\otimes\phi^{-1}(L)$, and we obtain
$x\in\phi^{-1}(H)\wedge\phi^{-1}(L)$.
\newline (4). %
For any $x\in{A\wedge{B}}$, we have $\Delta_C(x)\in{A\otimes{C}}+C\otimes{B}$, then
$\Delta_{\eC{C}}(\phi(x))=(\phi\otimes\phi)\Delta_C(x)\in\phi(A)\otimes\phi(C)+\phi(C)\otimes\phi(B)$.
\end{proof}

By applying the previous Lemma we may prove the following Proposition.

\begin{proposition}\label{pr:18}
Let $D$ be a prime subcoalgebra of $C$, then, for any non zero idempotent element
$e\in{C^\ast}$, we have $\eC{D}\subseteq\eC{C}$ is a prime subcoalgebra.
\end{proposition}
\begin{proof}
Let $H$, $L\subseteq\eC{C}$ be subcoalgebras such that $\eC{D}\subseteq{H\wedge{L}}$,
then
$$
 D\subseteq\phi^{-1}(\eC{D})\subseteq\phi^{-1}(H\wedge{L})\subseteq\phi^{-1}(H)\wedge\phi^{-1}(L).
$$
Since $D$ is prime, we obtain that either $D\subseteq\phi^{-1}(H)$ or
$D\subseteq\phi^{-1}(L)$. Hence either $\eC{D}\subseteq{H}$ or $\eC{D}\subseteq{L}$ and
$\eC{D}$ is prime.
\end{proof}

We would like to point out that primeness is a \textsl{local property}, in the sense
that in order to prove that a subcoalgebra $D$ of $PC(G)$ is prime we only need to
check
it for some special coalgebras, namely, those defined as $\eC{D}$ for certain
idempotent element
$e\in{PC(G)^*}$.

\begin{theorem}\label{th:42}
Let $D$ be a subcoalgebra of $PC(G)$. The following statements are equivalent:
\balf
\item
$D$ is prime;
\item
For any non zero idempotent element $e\in{PC(G)^\ast}$, defined as the characteristic
function of a set of two vertices, the coalgebra $\eC{D}\subseteq\eC{C}$ is prime.
\ealf
\end{theorem}
\begin{proof}
We only need to prove that (b) implies (a). Indeed, let $A$, $B\subseteq{PC(G)}$ be two
subcoalgebras such that $D\subseteq{A\wedge{B}}$. Let us assume that there is an
element
$x\in{D}\setminus{A}$. Furthermore, we may assume that $x$ is a linear combination of
paths with common source and common tail. Let $a=s(x)$ and $b=t(x)$. If we consider
$e:PC(G)\longrightarrow{k}$ the characteristic function of $\{a,b\}$, then we obtain
\[
 x=\eC{x}\in\eC{D}\subseteq\eC{(A\wedge{B})}\subseteq(\eC{A})\wedge^{\eC{C}}(\eC{B}).
\]
Since, by hypothesis, $\eC{D}$ is prime then either
$x\in\eC{D}\subseteq\eC{A}\subseteq{A}$, which is a contradiction, or
$x\in\eC{D}\subseteq\eC{B}\subseteq{B}$. Hence $D\subseteq{B}$ and $D$ is prime.
\end{proof}

As a consequence, in order to check whether a subcoalgebra $D$ of $PC(G)$ is prime, it
is enough to check if $\C{e}{D}\subseteq\C{e}{PC(G)}$ is prime for any non zero
idempotent $e\in{PC(G)}^\ast$ defined by a set of two vertices.

\end{document}